\documentclass[11pt]{amsart}
\usepackage{amsmath, amssymb, latexsym, amsthm,bm}
\usepackage{mathrsfs}
\usepackage{color}
\usepackage{graphicx}
\usepackage{mathscinet}

\usepackage{caption}
\usepackage{amsrefs}

\def\ignore#1{{ }}

\def\blueit#1{{\textcolor{blue}{#1}}}

\def\sphere{{S}}

\newtheorem{theorem}{Theorem} 
\newtheorem{theorem*}{Theorem} 
\newtheorem{proposition}[theorem]{Proposition} 

\newtheorem{lemma}[theorem]{Lemma}

\newtheorem{conjecture}[theorem]{Conjecture}
\newtheorem{observation}[theorem]{Observation}

\def\rat{{\leadsto}}

\def\dd{\mathscr D}

\begin{document}

\author{Carolina Medina}
\address{Department of Mathematics, University of California, Davis. Davis, CA 95616, USA}
\email{\tt cmedina@math.ucdavis.edu}
\thanks{The first author was supported by a Fulbright Visiting Scholar fellowship at UC Davis.}

\author{Gelasio Salazar}
\address{Instituto de F\'\i sica, Universidad Aut\'onoma de San Luis Potos\'{\i}. San Luis Potos\'{\i}, Mexico.}
\email{\tt gsalazar@ifisica.uaslp.mx}

\title[Crossing exchanges and smoothings on links]{{When can a link be obtained from another \\ using crossing exchanges and smoothings?}}

\begin{abstract}
Let $L$ be a fixed link. Given a link diagram $D$, is there a sequence of crossing exchanges and smoothings on $D$ that yields a diagram of $L$? We approach this problem from the computational complexity point of view. It follows from work by Endo, Itoh, and Taniyama that if $L$ is a prime link with crossing number at most $5$, then there is an algorithm that answers this question in polynomial time. We show that the same holds for all torus links $T_{2,m}$ and all twist knots.
\end{abstract}

\date{\today}


\maketitle

\ignore{
\def\gg{{\mathcal G}}
\def\pp{{\mathscr P}}
\def\ll{{\mathscr L}}
\def\oo{{\mathscr O}}
\def\ff{{\mathscr F}}
\def\uu{{\mathscr U}}
\def\ii{{\mathscr I}}
\def\ii{{\mathscr I}}
\def\hh{{\mathscr H}}
\def\hh{{\mathcal H}}
\def\gag{{\mathscr G}}
\def\bb{{\mathscr B}}
\def\ss{{\mathscr S}}
\def\ss{{\mathcal S}}
\def\dd{\mathcal D}
\def\ee{\mathcal E}
\def\rr{\mathcal R}
\def\ww{\mathcal W}
\def\qq{\mathcal Q}
\def\nn{\mathcal N}
\def\dd{\mathcal D}
\def\ee{\mathcal E}
\def\rr{\mathcal R}
\def\ww{\mathcal W}
\def\qq{\mathcal Q}
\def\nn{\mathcal N}
\def\aa{{\mathscr A}}
\def\aaa{{\mathscr A}}
\def\jj{{\mathscr J}}
\def\kk{{\mathscr K}}
\def\tll{{\mathscr L}}
\def\mm{{\mathscr M}}
\def\mm{{\mathcal M}}
\def\nn{{\mathscr N}}
\def\bb{{\mathscr B}}
\def\bb{{\mathcal B}}
\def\cc{{\mathscr C}}
\def\cc{{\mathcal C}}
\def\ee{{\mathscr D}}
\def\ee{{\mathscr E}}
\def\ee{{\mathcal E}}
\def\uu{{\mathscr U}}
\def\vv{{\mathscr V}}
\def\ww{{\mathscr W}}
\def\xx{{\mathscr X}}
\def\yy{{\mathscr Y}}
\def\zz{{\mathscr Z}}
}

\def\aa{{\mathscr A}}
\def\bb{{\mathscr B}}
\def\cc{{\mathscr C}}
\def\hh{{\mathscr H}}
\def\mm{{\mathscr M}}
\def\ss{{\mathscr S}}
\def\pp{{\mathscr P}}
\def\qq{{\mathscr qq}}
\def\gg{{\mathcal G}}
\def\ll{{\mathscr L}}

\section{Introduction}\label{sec:intro}

We work in the piecewise linear category. All links under consideration are nonsplit, unordered, unoriented and contained in the 3-sphere $\sphere^3$. We remark that when we speak of a link $L$ we include the possibility that $L$ is a link with only one component, that is, a knot. All diagrams under consideration are regular diagrams in the $2$-sphere $\sphere^2\subset \sphere^3$. 

This work revolves around the following basic question. Let $L$ be a fixed link. Given a link diagram $D$, does there exist a sequence of crossing exchanges and smoothings on $D$ that yields a diagram of $L$? If this is the case, then for brevity we write $D \,\rat L$.

\begin{figure}[ht!]
\centering
\scalebox{0.4}{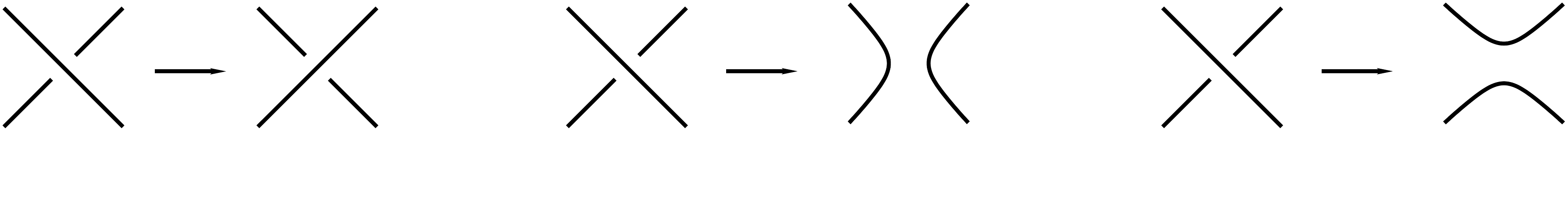}
\caption{In (a) we illustrate a crossing exchange operation, and in (b) and (c) the two crossing smoothing operations.}
\label{fig:Fig1}
\end{figure}

One wonders how ``difficult'' this question is. To formalize this, we need to work under the computational complexity setting, posing this question in the standard form of a decision problem. 
\medskip

\noindent{\bf Problem: }{\sc $\rat L$} (where $L$ is a fixed link)

\medskip

\noindent{\bf Instance: }{A link diagram $D$.}

\medskip

\noindent{\bf Question: }{Is it true that $D\,\rat L$?}

\medskip

We conjecture that for each fixed link $L$, the decision problem $\rat L$ is tractable, that is, there is a polynomial-time algorithm that solves $\rat L$. We recall that this means that there exist an algorithm $\aa(L)$ and a polynomial $p(n)$ such that the following holds. For each diagram $D$ with $n$ crossings, the algorithm $\aa(L)$ decides whether or not $D\,\rat L$ in at most $p(n)$ time steps. 

\begin{conjecture}\label{con:con1}
For each fixed link $L$, there is a polynomial time algorithm that solves $\rat L$.
\end{conjecture}

We start by noting that the available evidence backs up Conjecture~\ref{con:con1}. It is easy to see that it follows from results by Endo, Itoh, and Taniyama~\cite{endo} that Conjecture~\ref{con:con1} holds for each prime link with crossing number at most five:

\medskip
\noindent{\bf Theorem. }(Follows from~\cite[Theorems 2.4--2.10]{endo}) {\em Let $L$ be a fixed prime link with crossing number at most $5$. Then there is a polynomial time algorithm that solves $\rat L$.}
\medskip

\subsection{Our main result} 

We give further evidence to the plausibility of Conjecture~\ref{con:con1}, by showing that it holds for two important infinite classes of links, namely torus links $T_{2,m}$ and twist knots $T_m$. We remark that an important motivation to investigate crossing smoothing operations (an instance of band surgery) comes from current research in molecular biology. As explained in~\cite[Section 2.2]{mariel}, torus links $T_{2,m}$ and two-bridge links (such as twist knots) are especially relevant in the biological context for mechanistic reasons.

\begin{theorem}\label{thm:main1}
If $L$ is either a torus link or a twist knot, then there is a polynomial time algorithm that solves {$\rat L$}.
\end{theorem}

\begin{figure}[ht!]
\centering
\scalebox{0.58}{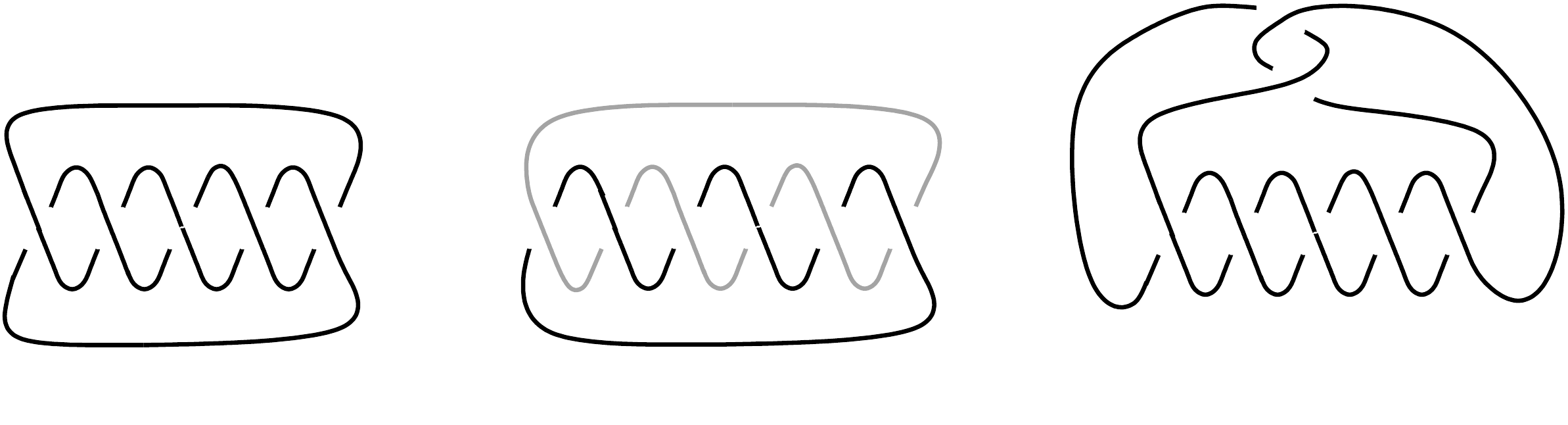}
\caption{In (a) we illustrate the torus knot $T_{2,5}$, and in (b) the torus link $T_{2,6}$. In general, $T_{2,m}$ has crossing number $m$, and we recall that $T_{2,m}$ is a knot if and only if $m$ is odd. In (c) we illustrate the twist knot $T_7$. For each integer $m\ge 3$, the twist knot $T_m$ has crossing number $m$.}
\label{fig:Fig57}
\end{figure}

\subsection{Related work}

Given a link $L$ and a diagram $D$, the question of whether or not $D\,\rat L$ arises in the context of defining an order in the collection $\ll$ of all links. If $L_1$ and $L_2$ are links, write $L_1\preceq L_2$ if every diagram of $L_2$ can be transformed into a diagram of $L_1$ by a (perhaps empty) sequence of crossing exchanges and smoothings. As shown in~\cite{endo}, the relation $\preceq$ (the {\em smoothing order}) is a pre-order in $\ll$, and it is a partial order on the set of all prime alternating links. 

In general, it is quite natural to ask which knots, or knot projections, are related under some set of local operations. We refer the reader to~\cite{thirtytwo} for a recent work in this theme.

\ignore{The question of whether or not two link diagrams are related by a sequence of crossing exchanges and/or smoothings is by itself a quite natural question and, \blueit{as we discuss below}, is particularly relevant nowadays given its relationship to current research in molecular biology.}

In~\cite{taniyama}, Taniyama gave characterizations of when a fixed knot $K$ with crossing number at most $5$ can be obtained from a diagram $D$ by a sequence of crossing exchanges. This was recently extended by Takimura~\cite{sixtwo} for the case in which $K$ is the knot $6_2$. In~\cite{taniyama2}, Taniyama gave characterizations for $2$-component links. We refer the reader to~\cite{fertility} and~\cite{hanaki1} for further results in this direction. A related notion investigated in~\cite{hanaki0,hanaki} is the trivializing number of a knot.

Crossing smoothings are an instance of band surgery operations. Another band surgery-related operation on diagrams is a {\em band-move} or $H(2)$-{\em move}~\cite{hoste}. In~\cite{abe}, Abe, Hanaki, and Higa introduced the {\em band-unknotting number} $u_b(K)$ of a knot $K$, which is the minimum number of band-moves required to deform a diagram of $K$ into a diagram of the unknot. We also refer the reader to~\cite{abek} and~\cite{kan1,kan2,kan3}.

Another question related to the main theme of our work is whether or not a given knot has a diagram for which smoothing a single crossing results in a diagram of its mirror image. This problem was recently investigated by Livingston~\cite{livingston} and by Moore and Vazquez~\cite{mariel2}.

\ignore{
\blueit{As we mentioned below, a motivation to investigate crossing exchanges and smoothings arises from current research in molecular biology. It is known that in the process of DNA recombination, a link is obtained from another by a certain type of crossing smoothing. Moreover, in experimental observations obtained in a laboratory sometimes only a projection of a link is known (that is, without over/under information). Thus, in order to develop a systematic framework to approach problems of interest to molecular biologists, it makes sense to start with the simplest possible question, which is precisely to understand the (structural or computational) complexity of deciding whether a link can obtained from another by a sequence of crossing exchanges and/or smoothings.}
}

\section{Reducing Theorem~\ref{thm:main1} to large and sufficiently connected diagrams}\label{sec:red1}

Our aim in this section is to show that it suffices to prove Theorem~\ref{thm:main1} for the case in which the input diagram $D$ has ``many'' crossings, and satisfies a certain connectivity property. 

\ignore{The precise meaning of ``sufficiently large'' is the following. Suppose that the fixed link $L$ has crossing number $m$. We will show that there is an integer $n:=n(m)$ such that the following holds. If Theorem~\ref{thm:main0} (respectively, Theorem~\ref{thm:main1}) holds when the input diagram $D$ has crossing number at least $n$, then it holds for every input diagram $D$.}

We recall that a {\em projection} $P$ is obtained from a diagram $D$ by omitting from $D$ the over/under information at each crossing. We say that $P$ {\em is the projection of $D$}.

For an illustration of the next notion we refer the reader to Figure~\ref{fig:Fig60}. {We say that a projection is {\em strong} if there is no simple closed curve $\gamma$ such that (i) $\gamma$ intersects $P$ in exactly two points, which are noncrossing points of $P$; and (ii) each component of $\sphere^2\setminus \gamma$ contains at least one crossing point of $P$.}

\ignore{
\blueit{We say that a pair $\{x,y\}$ of noncrossing points $x,y$ of a projection $P$ is {\em distant} if, when we regard $P$ as a four-valent plane graph, $x$ and $y$ do not belong to the same edge of $P$.} A projection $P$ is \blueit{{\em strong}} if for any pair $x,y$ of \blueit{distant} noncrossing points of $P$, we have that $P\setminus\{x,y\}$ is connected. Note that, in particular, if $P$ has a crossing point $z$ such that $P\setminus\{z\}$ is disconnected, then $P$ is not strong, as then there exist noncrossing points $x,y$ very close to $z$ whose removal disconnects $P$. 
}

\begin{figure}[ht!]
\centering
\scalebox{0.9}{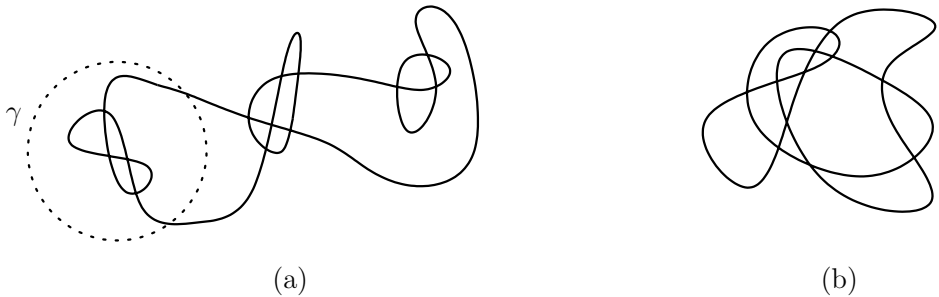}
\caption{The curve $\gamma$ witnesses that the projection in (a) is not strong. It is easy to verify that the projection in (b) is strong.}
\label{fig:Fig60}
\end{figure}

A diagram is {\em strong} if its projection is strong. We note that every crossing-minimal diagram of a  prime link is strong.

The following statement is quite useful, as it will allow us to prove Theorem~\ref{thm:main1} under the assumption that the input diagrams of the decision problem $\rat L$ are strong and ``large''.

\begin{proposition}\label{pro:pro1}
Let $L$ be a fixed prime link, and let $n_0$ be a constant. Suppose that there is a polynomial time algorithm that solves $\rat L$ under the assumption that the input diagram $D$ is strong and has at least $n_0$ crossings. Then there is a polynomial time algorithm that solves $\rat L$.
\end{proposition}

\begin{proof}
Let $L$ be a fixed prime link. We start by supposing that there is a polynomial time algorithm that solves $\rat L$ when restricted to strong diagrams. That is, there is an algorithm $\aa(L)$ and a polynomial $p(n)$ such that, when the input diagram $D$ is strong, $\aa(L)$ decides whether or not $D\,\rat L$ in at most $p(n)$ time steps, where $n$ is the number of crossings in $D$. We will show that then there is a polynomial time algorithm that solves $\rat L$. 

Let $D$ be an arbitrary  diagram with $n$ crossings. If $D$ is not strong, then we recursively decompose $D$ by using the cut-and-repair operation illustrated in Figure~\ref{fig:Fig63}, until we finally obtain a collection $D_1,D_2,\ldots,D_k$ of strong subdiagrams of $D$, where $D_i$ has $n_i$ crossings for each $i=1,\ldots,k$, and $\sum_{i=1}^k n_i  = n$. This decomposition of $D$ into strong subdiagrams can clearly be performed in a number of steps bounded by a polynomial function of $n$.

\begin{figure}[ht!]
\centering
\scalebox{0.9}{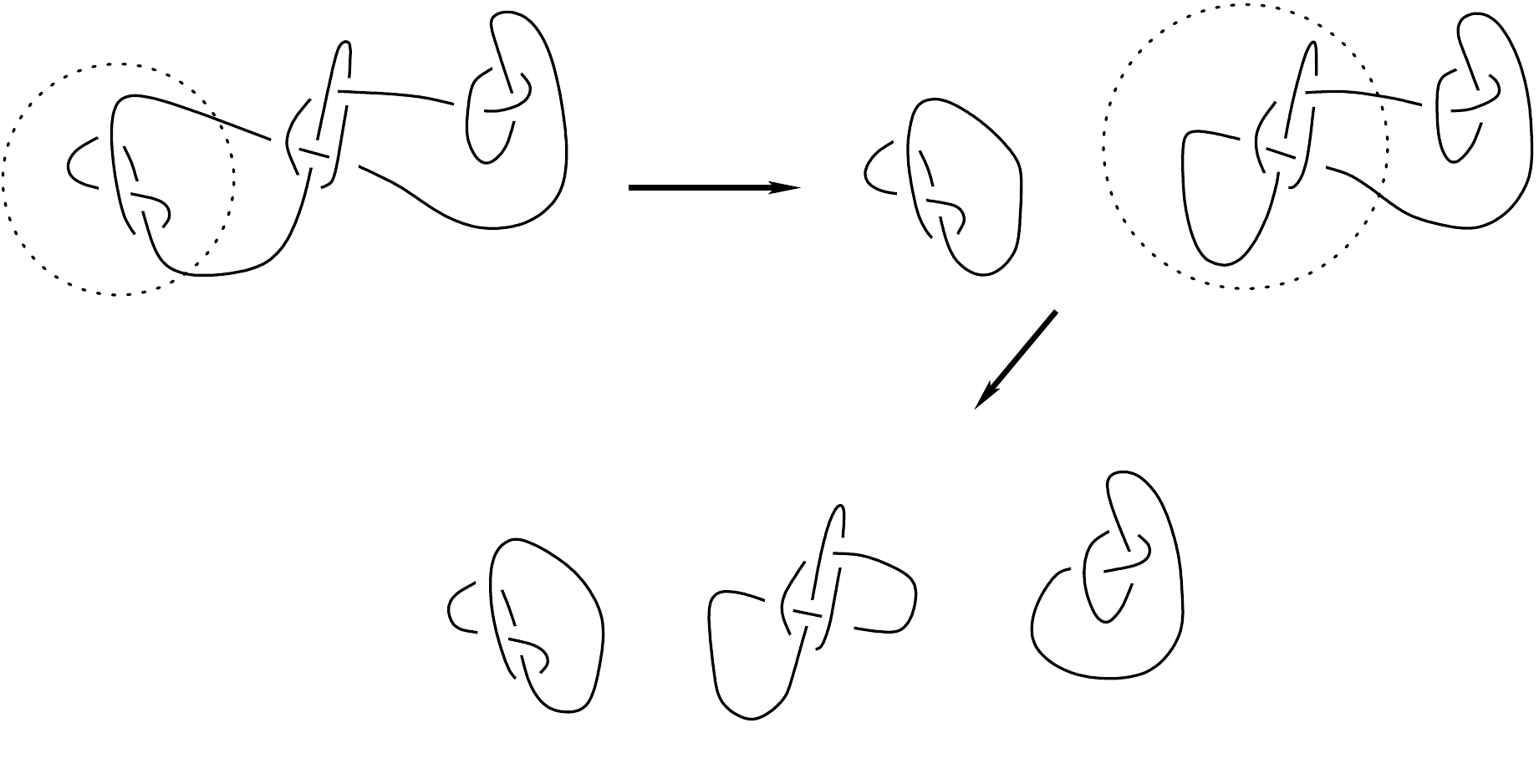}
\caption{Illustration of the proof of Proposition~\ref{pro:pro1}: decomposing a diagram $D$ into strong sub-diagrams.}
\label{fig:Fig63}
\end{figure}

The key observation is that, since $L$ is prime (in Theorem~\ref{thm:main1}, $L$ is either a torus link or a twist knot), then $D\,\rat L$ if and only if there is an $i\in\{1,\ldots,k\}$ such that $D_i\,\rat L$. We now apply the algorithm $\aa(L)$ to $D_i$ for $i=1,\ldots,k$. The amount of time steps required to check whether or not there is an $i\in\{1,\ldots,k\}$ such that $D_i\,\rat L$, and hence to decide whether or not $D\,\rat L$, is then at most $\sum_{i=1}^k p(n_i)$. 

Every polynomial of degree at least $1$ is superadditive, and so if $p(n)$ has degree at least $1$, then this amount of time is at most $p(n)$, and so we are done. In the alternative, $p(n)$ is a constant $c$. In this case, the amount of time is at most $\sum_{i=1}^k p(n_i)=\sum_{i=1}^k c= c\cdot k \le c\cdot n$.

We have proved that if there is a polynomial time algorithm that solves $\rat L$ when the input diagram is strong, then there is a polynomial time algorithm that solves $\rat L$ for an arbitrary input diagram. 

Suppose now that there is a polynomial time algorithm that solves $\rat L$ when the input diagram is strong and has at least $n_0$ crossings. We will show that then there is a polynomial time algorithm that solves $\rat L$ when the input diagram is strong. In view of the previous paragraph, this will prove that there is a poynomial time algorithm that solves $\rat L$.

Thus the assumption is that there is an algorithm $\aa'(L)$, and a polynomial $q(n)$, such that, when the input diagram $D$ is strong and has at least $n_0$ crossings, $\aa'(L)$ decides whether or not $D\,\rat L$ in at most $q(n)$ time steps, where $n$ is the number of crossings in $D$. 

We now remark that for each fixed link $L$, the problem $\rat L$ is decidable. Indeed, let $D$ be any input diagram. Let $\dd(D)$ be the collection of all link diagrams that can be obtained from $D$ by crossing exchanges and smoothings. Thus $D\,\rat L$ can be decided in a finite number of time steps, since (i) $\dd(D)$ is finite; (ii) $L$ is a fixed link; and (iii) the problem of deciding whether or not two links are equivalent is decidable (see for instance~\cite{coward}). Let $\aa''(L)$ be this described algorithm that solves $\rat L$.

For each fixed strong diagram $D$ with at most $n_0$ crossings, let $m(D)$ be the amount of time steps that $\aa''(L)$ takes to decide whether or not $D\,\rat L$. Now let $M:=\max\{m(D)\}$, where the maximum is taken over all strong diagrams with at most $n_0$ crossings. Since there is a finite number of strong diagrams with at most $n_0$ crossings, it follows that $M$ is a well-defined constant. That is, since $n_0$ is fixed, the running time of $\aa''(L)$ is bounded by an absolute constant $M$, when the input is restricted to strong diagrams with at most $n_0$ crossings.

To conclude the proof, let $\overline{\aa(L)}$ be the algorithm that results by combining $\aa'(L)$ and $\aa''(L)$: if the input strong diagram $D$ has at least $n_0$ crossings, then apply $\aa'(L)$ to $D$, and otherwise apply $\aa''(L)$. Then $\overline{\aa(L)}$ is a polynomial time algorithm, since it decides whether or not $D\,\rat L$ in at most $\max\{q(n),M\}< q(n)+M$ time steps.
\end{proof}

\section{Proof of Theorem~\ref{thm:main1}}\label{sec:red2}

The workhorses behind the proof of Theorem~\ref{thm:main1} are the following lemmas, which provide structural characterizations of when a (large) strong diagram $D$ satisfies that $D\,\rat L$, for the cases in which $L$ is a torus link or a twist knot, respectively.

We defer the proofs of these two lemmas for the moment, and give the proof of Theorem~\ref{thm:main1}. The rest of the paper is devoted to the proofs of these lemmas.

\begin{lemma}\label{lem:torus}
For each integer $m\ge 3$, there is an integer $n_1:=n_1(m)$ with the following property. If $D$ is any strong diagram with at least ${n_1}$ crossings, then $D\,\rat T_{2,m}$.
\end{lemma}

\begin{lemma}\label{lem:twist}
For each integer $m\ge 3$, there is an integer $n_2:=n_2(m)$ with the following property. Let $D$ be any strong diagram with at least $n_2$ crossings. Then $D\,\rat T_{m}$ if and only if the projection $P$ of $D$ is not the projection of a crossing-minimal diagram of a torus link $T_{2,n}$.
\end{lemma}

\begin{proof}[Proof of Theorem~\ref{thm:main1}, assuming Lemmas~\ref{lem:torus} and~\ref{lem:twist}]
Suppose first that $L$ is a torus link $T_{2,m}$, and let $n_1$ be as in Lemma~\ref{lem:torus}. By Proposition~\ref{pro:pro1} it suffices to consider the decision problem $\rat L$ restricted to strong diagrams with at least $n_1$ crossings. By Lemma~\ref{lem:torus}, the decision problem $\rat L$ restricted to these diagrams is evidently solved in polynomial time, as the answer is simply ``yes'' for all such diagrams.

Finally, suppose that $L$ is a twist knot $T_m$, and let $n_2$ be as in Lemma~\ref{lem:twist}. By Proposition~\ref{pro:pro1}, it suffices to consider $\rat L$ restricted to strong diagrams with at least $n_2$ crossings. By Lemma~\ref{lem:twist}, such a diagram $D$ satisfies that $D\,\rat L$ if and only if the projection $P$ of $D$ is not the projection of a crossing-minimal diagram of a torus link $T_{2,n}$. It is easy to see that whether or not $P$ satisfies this last property can be answered in polynomial time (say, from the Gauss code of $P$), and so we are done.
\end{proof}

\section{The key tool for the proofs of Lemmas~\ref{lem:torus} and~\ref{lem:twist}: Tait graphs and graph minors}

With the work in Sections~\ref{sec:red1} and~\ref{sec:red2}, we have reduced Theorem~\ref{thm:main1} to Lemmas~\ref{lem:torus} and~\ref{lem:twist}, and thus our remaining goal is to prove these lemmas. A crucial tool behind the proofs of these lemmas is a powerful result established in~\cite{endo}, relating Tait graphs of link projections to the question of whether or not $D\,\rat L$, for a given link $L$ and a given diagram $D$. 

\subsection{Tait graphs}

The notion of a Tait graph is illustrated in Figure~\ref{fig:Fig6}. 

\begin{figure}[ht!]
\scalebox{1.0}{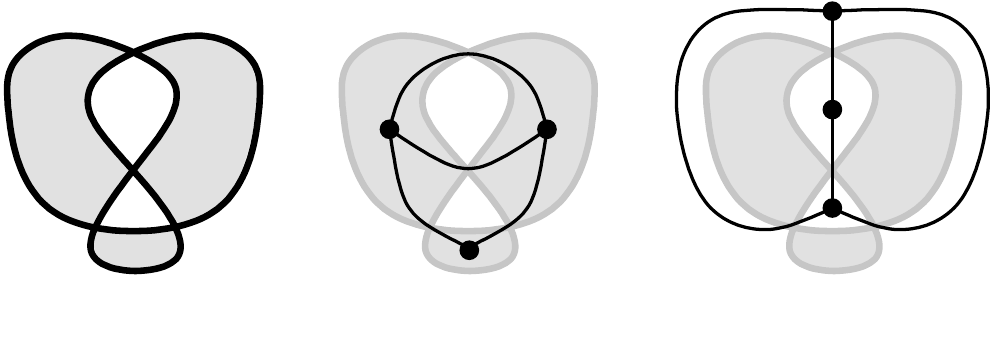}
\caption{On the left hand side of this figure we illustrate the projection of a crossing-minimal diagram of a figure-eight knot. In (b) and (c) we illustrate the corresponding Tait graphs.}
\label{fig:Fig6}
\end{figure}

Let $D$ be a diagram, and let $P$ be its projection. As explained in~\cite{hararykauffman}, we start by performing a checkerboard (gray and white) colouring of the faces of $P$. From the gray faces we obtain one Tait graph, as illustrated in Figure~\ref{fig:Fig6}(b), and from the white faces we obtain the other Tait graph, as illustrated in (c). Note that these plane graphs are dual of each other and, although not in our cases of interest, they may actually be the same graph. We remark that the number of crossings of $P$ is the number of edges of each of its Tait graphs.

We recall that a graph $H$ is a {\em minor} of a graph $G$ if $H$ is a subgraph of a graph obtained by performing edge contractions on $G$. In our current setting of plane graphs, we remark that all edge contractions are performed on {$\sphere^2$}, so that they respect the embedding of the graph on which we perform the contractions. Also, a subgraph of a plane graph is obtained by removing edges and/or vertices, without altering the embedding of the remaining edges and vertices.

\subsection{The key tool}

The following statement is the workhorse behind the proofs of Lemmas~\ref{lem:torus} and~\ref{lem:twist}.

\begin{proposition}[{\cite[Proposition 1.7]{endo}}]\label{pro:17}
Let $L$ be a link, and let $P_L$ be a projection of $L$. Let $D$ be a diagram, and let $P$ be its projection. Suppose that there is a Tait graph of $P_L$ that is a minor of a Tait graph of $P$. Then $D\,\rat L$.
\end{proposition}

The difficult direction of Lemmas~\ref{lem:torus} and~\ref{lem:twist} is that we have a link $L$, and we need to understand the structure of diagrams $D$ such that $D\,\rat L$. Proposition~\ref{pro:17} is then quite useful: if a Tait graph of the projection of $D$ contains as a minor a Tait graph of a projection $P_L$ of $L$, then we know that $D\,\rat L$. 

Therefore it is very valuable, given a Tait graph $T$, to understand which graphs contain (equivalently, which graphs do {\em not} contain) $T$ as a minor. This structural characterization (for the Tait graphs of torus links and twist knots) will be the key ingredient in the proofs of Lemmas~\ref{lem:torus} and~\ref{lem:twist}.

To prepare the terrain towards this goal, we finish this section with a few basic graph theory notions, and state an elementary graph theory result that will be very useful.

\subsection{Basic graph theory notions and an auxiliary result}

Let $G$ be a graph, and let $C$ be a cycle of $G$. A {\em chord} of $C$ is an edge whose endvertices that are not adjacent in $C$. A $C$-{\em path} is a path whose endvertices are in $C$, and is otherwise disjoint from $C$. Note that a chord is a particular kind of $C$-path. 

We will make use of the following elementary graph theory fact. We recall that the {\em circumference} of a graph is the length of a longest cycle. As usual, $G^*$ denotes the dual of a plane graph $G$.

\begin{observation}\label{obs:planar}
For each integer $k\ge 2$, there is an integer $k_0$ with the following property. If $G$ is a $2$-connected plane graph with at least $k_0$ edges, then either $G$ or $G^*$ has circumference at least $k$.
\end{observation}

\ignore{
\begin{proof}
It suffices to show that either $G$ or $G^*$ has a cycle of size at least $k$ as a minor. We start by noting that if $G$ has a vertex $v$ of degree $k$ or larger we are done. To see this, note that in this case the minor of $G$ obtained by contracting all edges not incident with $v$ is a bond with at least $k$ edges, and thus the dual of this bond is a cycle with at least $k$ edges that is a minor of $G^*$.
Thus we may assume that all vertices of $G$ have degree at most $k-1$. Let $C$ be a longest cycle in $G$. If $C$ has at least $k$ edges then we are done, so we suppose that $C$ has $r<k$ edges. We complete the proof by showing that if the number of edges of $G$ is sufficiently large, then there exists a $C$-path $Q$ with at least $k-1$ edges. This indeed completes the proof, as this implies the existence of a cycle in $G$ with at least $k$ edges.
Since each vertex of $G$ has degree at most $k-1$, if the number of edges of $G$ is large enough then there exists a vertex $u$ in $G$ whose distance (length of shortest path) to each vertex in $C$ is at least $(k-1)/2$. The $2$-connectedness of $G$ implies that there is a $C$-path $Q$ that contains $u$ as an internal vertex. Since the distance of $u$ to each vertex in $C$ is at least $(k-1)/2$, it follows that $Q$ has length at least $k-1$, as required.
\end{proof}
}

\section{Proof of Lemma~\ref{lem:torus}}

As illustrated in Figure~\ref{fig:Fig52}, it is easy to verify that if $P$ is the projection of a crossing-minimal diagram of the torus link $T_{2,m}$, then the Tait graphs of $P$ are the cycle $C_m$ and the bond $B_m$. We recall that the bond $B_m$ is the two-vertex graph with $m$ parallel edges joining them.

\begin{figure}[ht!]
\centering
\scalebox{0.7}{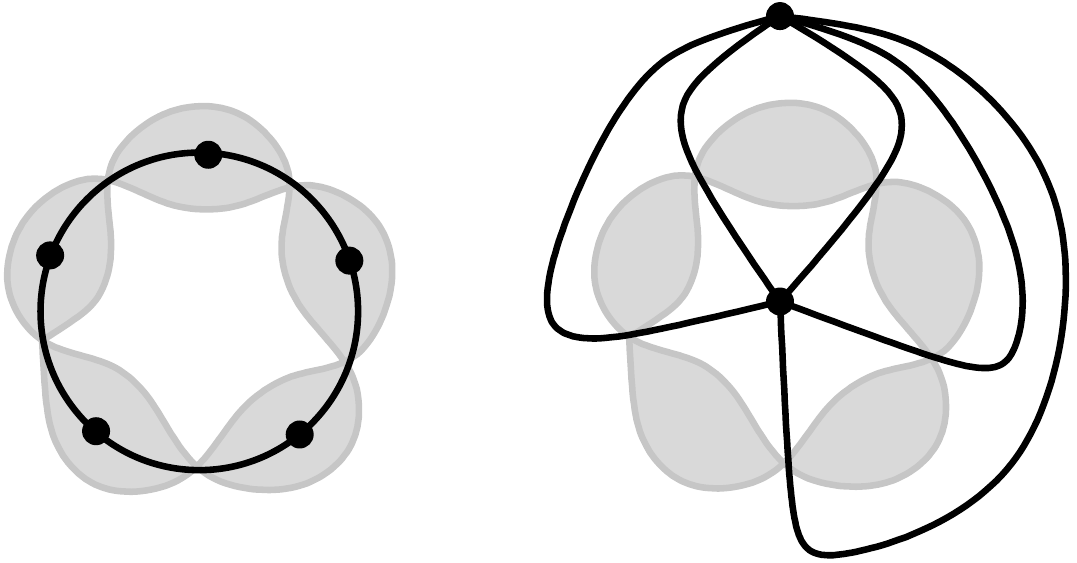}
\caption{The Tait graphs of the torus link $T_{2,5}$ are the $5$-cycle $C_5$ and the $5$-bond $B_m$.}
\label{fig:Fig52}
\end{figure}

\begin{proposition}\label{pro:torus}
For each integer $m\ge 3$, there is an integer $n_1:=n_1(m)$ with the following property. If $G$ is a $2$-connected plane graph with at least $n_1$ edges, then $G$ contains either $C_m$ or $B_m$ as a minor.
\end{proposition}

\begin{proof}
We start by noting that a graph has $C_m$ as a minor if and only if its dual has $B_m$ as a minor. Thus it suffices to show that either $G$ or $G^*$ has $C_m$ as a minor.

By Observation~\ref{obs:planar}, if the number of edges of $G$ is sufficiently large, then there is an $H\in\{G,G^*\}$ that has a cycle of length at least $m$. From this it follows that either $G$ or $G^*$ has $C_m$ as a minor.
\end{proof}

\begin{proof}[Proof of Lemma~\ref{lem:torus}]
Let $m\ge 3$ be an integer. Let $D$ be a strong diagram with at least $n_1$ edges, where $n_1$ is as in Proposition~\ref{pro:torus}. Let $P$ be the projection of $D$, and let $G$ be a Tait graph of $P$. Since $D$ is strong and has at least $n_1$ crossings, then the same holds for $P$, and so $G$ has at least $n_1$ edges and is $2$-connected. By Proposition~\ref{pro:torus}, $G$ contains either $C_m$ or $B_m$ as a minor. Since $C_m$ and $B_m$ are the Tait graphs of a crossing-minimal diagram of $T_{2,m}$, by Proposition~\ref{pro:17} it follows that $D\,\rat T_{2,m}$.
\end{proof}

\section{Proof of Lemma~\ref{lem:twist}}

In Figure~\ref{fig:Fig40}(a) and (b) we show a projection $P$ of a crossing-minimal diagram of the twist knot $T_6=6_1$, and the corresponding Tait graphs of $P$. In (c) and (d) we draw these two Tait graphs in a more visually appealing way. One of these Tait graphs is a $5$-cycle plus a parallel edge added to one of its edges (we call this graph $C_5^+$), and the other Tait graph is a $5$-bond with one edge subdivided exactly once (we call it $B_5^+$).  

In general, if $P$ is the projection of a crossing-minimal diagram of the twist knot $T_m$, then the Tait graphs of $P$ are (i) the graph $C_{m-1}^+$ obtained by taking a cycle $C_{m-1}$ and adding exactly one parallel edge; or (ii) the graph $B_{m-1}^+$ obtained by taking a bond $B_{m-1}$ and subdividing one edge exactly once.

\begin{figure}[ht!]
\centering
\scalebox{0.5}{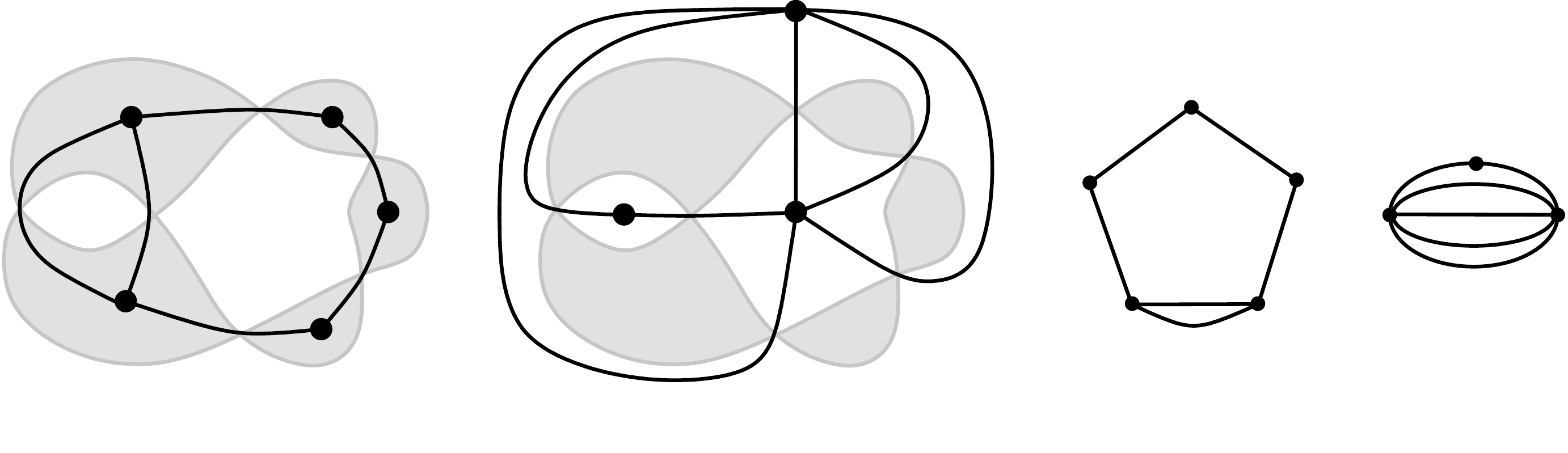}
\caption{The Tait graphs of the twist knot $T_{6}=6_1$ are the graph $C_5^+$ obtained from a $5$-cycle by adding a parallel edge to one edge, as shown in (c), and its dual $B_5^+$, illustrated in (d), which is obtained from a $5$-bond by subdividing one edge exactly once.}
\label{fig:Fig40}
\end{figure}

\begin{proposition}\label{pro:twist}
For each integer $m\ge 3$, there is an integer $n_2:=n_2(m)$ with the following property. Let $G$ be a $2$-connected plane graph with $n\ge n_2$ edges, that contains neither $C_{m-1}^+$ nor $B_{m-1}^+$ as a minor. Then $G$ is either the cycle $C_n$ or the bond $B_n$.
\end{proposition} 

\begin{proof}
Since $C_{m-1}^+$ is the dual of $B_{m-1}^+$, the assumption implies that $G^*$ contains neither $C_{m-1}^+$ nor $B_{m-1}^+$ as a minor. Since the dual of $C_n$ is $B_n$, it suffices to show that one of $G$ and $G^*$ is $C_n$. 

Let $n_2$ be an integer such that every plane $2$-connected graph satisfies that either it or its dual contains a cycle of size at least $2(m-2)$. The existence of $n_2$ is guaranteed by Observation~\ref{obs:planar}. Suppose that $G$ has $n\ge n_2$ edges. Then there is an $H\in\{G,G^*\}$ that has a cycle $C$ with at least {$2(m-2)$} edges. 

If $H=C$ then $H=C_n$, and we are done. Suppose then that $H\neq C$. Then there must exist a $C$-path $Q$. Let $u,v$ be the endvertices of $Q$. Let $Q_1$ and $Q_2$ be the paths in $C$ that have $u$ and $v$ as endvertices, labelled so that $Q_1$ has at least as many edges as $Q_2$. See Figure~\ref{fig:Fig45}(a). Thus $Q_1$ has at least $m-2$ edges.  As illustrated in (b), then $C\cup Q$, and hence $G$, has $C_{m-1}^+$ as a minor, a contradiction.
\end{proof}

\begin{figure}[ht!]
\centering
\scalebox{0.7}{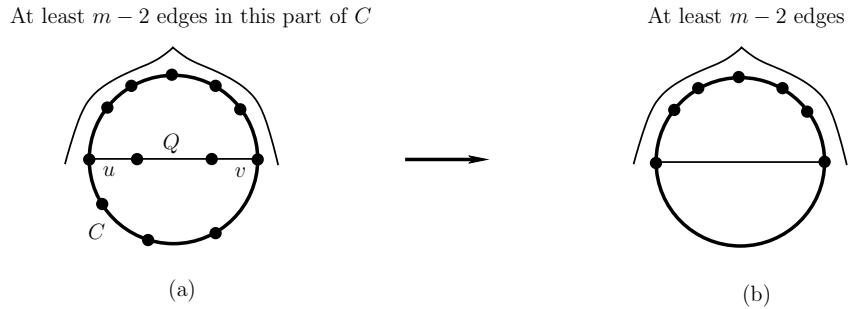}
\caption{Illustration of the proof of Proposition~\ref{pro:twist}.}
\label{fig:Fig45}
\end{figure}

\begin{proof}[Proof of Lemma~\ref{lem:twist}]
Let $m\ge 3$ be an integer. Let $D$ be a strong diagram with at least $n_2$ edges, where $n_2:=n_2(m)$ is as in Proposition~\ref{pro:twist}. Let $P$ be the projection of $D$. It is easy to see that if $P$ is the projection of a crossing-minimal diagram of a torus link $T_{2,n}$ then $D\not\leadsto T_m$. Thus it only remains to prove that if $P$ is {\em not} the projection of a crossing-minimal diagram of a torus link $T_{2,n}$, then $D\,\rat T_m$. Note that this assumption on $P$ means that none of the Tait graphs of $P$ is a cycle or a bond.

Let $G$ be a Tait graph of $P$. Since $D$ is strong and has at least $n_2$ crossings, then the same holds for $P$, and so $G$ has at least $n_2$ edges and is $2$-connected. Since $G$ is neither a cycle nor a bond, it follows from (the contrapositive of) Proposition~\ref{pro:twist} that $G$ contains either $C_{m-1}^+$ or $B_{m-1}^+$ as a minor. Since these are the Tait graphs of a projection of $T_m$, it follows from Proposition~\ref{pro:17} that $D\,\rat T_m$.
\end{proof}

\section{An open question}\label{sec:concluding}

If Conjecture~\ref{con:con1} turns out to be true, then the next natural step would be to consider the decision problem in which the link $L$ is not fixed, but it is part of the input:

\medskip

\noindent{\bf Problem: }{\sc $\rat$}

\medskip

\noindent{\bf Instance: }{A link $L$, and a link diagram $D$.}

\medskip

\noindent{\bf Question: }{Is it true that $D\,\rat L$?}

\medskip

What is the computational complexity of this decision problem?

\end{document}